\documentclass[12pt]{amsart}

\usepackage{epsfig}
\usepackage{amsmath}
\usepackage{amssymb,amsthm}
\usepackage{graphicx}
\usepackage{booktabs}

\usepackage{pst-node}
\usepackage{tikz}
\usepackage{enumerate}
\usepackage[backref]{hyperref}
\usepackage{a4wide}
\usepackage[normalem]{ulem}

\usepackage[margin=2.9cm]{geometry}

\newtheorem{propo}{Proposition}[section]

\newtheorem{lemma}[propo]{Lemma}

\newtheorem{theo}[propo]{Theorem}
\newtheorem{examp}[propo]{Example}

\newcommand{\bl}{\begin{lemma}}
\newcommand{\el}{\end{lemma}}

 \def\C{{\bf C}}

\usepackage{indentfirst,latexsym,bm}

\def\C{{\rm C}}

\def\K{{\rm K}}

\begin{document}
\title{A classification of locally-quasiprimitive circulant digraphs}

\thanks{Supported by NSFC (12271524, 12331013), NSF of Hunan (2026JJ50358) and the Scientific Research Project of the Education Department of Hunan Province (24A0142).}

\author[W. Jin]{Wei Jin}
\address{Wei Jin\\
School of Mathematics and Computational Science\\
Hunan Research Center of the Basic Discipline Fundamental Algorithmic Theory and Novel Computational Methods\\
Xiangtan University\\
Xiangtan, Hunan, 411105, P.R. China}
\address{School of Statistics and Data Science\\
Jiangxi University of Finance and Economics\\
Nanchang, Jiangxi, 330013, P.R. China}
\email{jinweipei82@163.com}

\author[Y.X. Jin]{Yu Xiang Jin}
\address{Yu Xiang Jin}
\address{School of Mathematics and Computer Sciences\\
Nanchang University\\
Nanchang, Jiangxi, 330031, P.R.China}
\email{18779893910@163.com}

\author[C. X. Li]{Cai Xia Li}
\address{Cai Xia Li (Corresponding author)\\
School of Mathematics and Computational Science,\\
Key Laboratory of Intelligent Computing and Information Processing of Ministry of Education\\
Xiangtan University\\
Xiangtan, Hunan, 411105, P.R. China}
\email{caixial@yeah.net}

\author[P. S. Li]{Ping Shan Li}
\address{Ping Shan Li\\
School of Mathematics and Computational Science\\
Key Laboratory of Intelligent Computing and Information Processing of Ministry of Education\\
Xiangtan University\\
Xiangtan, Hunan, 411105, P.R. China}
\email{lips@xtu.edu.cn}

\date\today

\maketitle

\begin{abstract}

Circulant digraphs are Cayley digraphs over finite cyclic groups and constitute a fundamental class of objects in algebraic graph theory. Extending the  classification of locally-primitive circulant graphs \cite{JZ-2026}, we completely determine all locally-quasiprimitive circulant digraphs.
Our main theorem shows that  a connected locally-quasiprimitive circulant digraph  is isomorphic to one of the following:
the complete graph \(\K_n\),
the complete bipartite graph \(\K_{n/2,n/2}\),
the graph   \(\K_{n/2,n/2}-\frac{n}{2}\K_2\) (with \(n/2\) odd),
the cycle \(\C_n\),
the directed cycle \(\vec \C_n\),
a normal circulant digraph of prime valency,
the lexicographic product \(\vec \C_m[\overline{\K_b}]\),  or
the tensor product \(\vec \C_m\times \K_b\) with \(\gcd(m,b)=1\).


\end{abstract}

\vspace{2mm}

\section{Introduction}

In this paper, we establish a complete classification of locally-quasiprimitive circulant digraphs, a prominent class of symmetric digraphs in algebraic graph theory. We begin by introducing the fundamental definitions, standard notation, and necessary research background to frame our investigation.

A finite \emph{digraph} (short for \emph{directed graph}) $\Gamma$ is a pair $(V(\Gamma),\operatorname{Arc}(\Gamma))$ consisting of a finite vertex set $V(\Gamma)$ and an arc set $\operatorname{Arc}(\Gamma)\subseteq V(\Gamma)\times V(\Gamma)$. An element $(u,v)\in \operatorname{Arc}(\Gamma)$ is a \emph{directed edge} from $u$ to $v$, denoted $u\rightarrow v$; the ordered pair $(u,v)$ is then called an \emph{arc} of $\Gamma$. For a vertex $v\in V(\Gamma)$, the set of \emph{in-neighbours} of $v$ is $\Gamma^-(v)=\{u\in V\mid u\rightarrow v\}$, and the set of \emph{out-neighbours} of $v$ is $\Gamma^+(v)=\{u\in V\mid v\rightarrow u\}$. A digraph $\Gamma$ is called \emph{$k$-regular} if $|\Gamma^-(v)|=|\Gamma^+(v)|=k$ for every $v\in V(\Gamma)$, and $\Gamma$ is \emph{regular} if it is $k$-regular for some positive integer $k$, this integer $k$ is then the \emph{valency} of $\Gamma$. A digraph $\Gamma$ is said to be \emph{locally-quasiprimitive} if, for each vertex $v$, the stabilizer of $v$ in the full automorphism group acts quasiprimitively on the set of out-neighbours of $v$. A \emph{graph} is a digraph in which every arc is bidirectional: for all vertices $u,v$, $(u,v)$ is an arc exactly when $(v,u)$ is an arc. If $\Gamma$ is a graph and $v\in V(\Gamma)$, then $\Gamma^+(v)=\Gamma^-(v)$, and we denote this common set by $\Gamma(v)$.

A digraph (graph) $\Gamma$ is called a \emph{circulant digraph} (\emph{circulant graph}) if its full automorphism group contains a cyclic subgroup that acts semiregularly and transitively on the vertex set. As one of the most fundamental and extensively studied families of algebraic graphs, circulant digraphs (graphs) possess inherent cyclic symmetry and highly regular structural properties. Owing to their elegant combinatorial features and wide-ranging applications in network design, coding theory and parallel computing, they have long been a central object in the study of vertex-transitive digraphs (graphs).

The investigation of  circulant digraphs naturally began with primitive group actions. By a classical theorem of Schur (see \cite[Theorem~25.2]{Wielandt-book}), any circulant graph whose automorphism group acts primitively on the vertex set is either a complete graph or a circulant graph of prime order. This milestone result laid a solid foundation for subsequent work and motivated the study of circulant graphs under refined group-theoretic constraints.

While primitivity is a strong global symmetry condition, later research shifted to the more flexible property of arc-transitivity. The systematic study of arc-transitive circulant digraphs was initiated in the 1970s by Chao and Wells \cite{Chao1971,Chao1973}, who pioneered the exploration of symmetric circulant structures. Subsequent advances greatly enriched this field: Alspach et al.\ \cite[Theorem~1.1]{ACMX-1996} completed the classification of $2$-arc-transitive circulant graphs, while Kov\'acs \cite{Kovacs-2004} and Li \cite{LCH-circulant-2005} independently established reductive characterisations of connected arc-transitive circulant digraphs. Recently, Li, Xia and Zhou \cite{LXZ-at-circ-2021} refined these results into an explicit structural characterization of arc-transitive circulant digraphs.
The classification of finite locally-primitive circulant graphs was   completed in \cite{JZ-2026}.

Despite the extensive progress on primitive, 2-arc-transitive and locally-primitive circulant digraphs, the classification under the weaker yet critical condition of local quasiprimitivity remains unresolved. Local quasiprimitivity is a vital group-theoretic symmetry property that generalizes primitivity and captures finer local structural features. Li, Praeger, Venkatesh and Zhou \cite{LPVZ-2002} studied normal quotients of locally quasiprimitive graphs and demonstrated the relevance of this class. As a natural and meaningful extension of prior classification work, characterizing locally-quasiprimitive circulant digraphs not only enriches the classification theory of symmetric circulant digraphs but also yields new insights into the interplay between local group actions and global graph symmetry. In this paper, we  establish a complete classification of locally-quasiprimitive circulant digraphs.

Our main theorem is the following:

\begin{theo}\label{localquasiprim-circ-th}
Let $T$ be a cyclic group of order $n\geq 3$ and let $S\subseteq T$ be such that $T=\langle S\rangle$. Let $\Gamma=\operatorname{Cay}(T,S)$ be a locally-quasiprimitive circulant digraph. Then  one of the following holds.

\begin{enumerate}[(I)]
\item $S=S^{-1}$, and $\Gamma$ is isomorphic to one of the following graphs:
\begin{itemize}
\item[(1)] $\K_n$;
\item[(2)] $\K_{\frac{n}{2},\frac{n}{2}}$;
\item[(3)] $\K_{\frac{n}{2},\frac{n}{2}}-\frac{n}{2}\K_2$, where $\frac{n}{2}$ is odd;
\item[(4)] the cycle  $C_n$.
\end{itemize}
In particular, $\Gamma$ is $2$-arc-transitive.

\item $S^{-1}\cap S=\varnothing$, and $\Gamma$ is isomorphic to one of the following digraphs:
\begin{itemize}
\item[(1)] the directed cycle $\vec{C}_n$;
\item[(2)] $|S|=p\ge 2$ is a prime,  and there exist an element $b\in T\setminus\{1\}$ and an automorphism $\mu\in\operatorname{Aut}(T)$ of order $p$ such that
\[
S=b^{\langle\mu\rangle}=\{\,b,b^\mu,b^{\mu^2},\dots,b^{\mu^{p-1}}\,\},
\]
$\langle S\rangle=T$, and $S\cap S^{-1}=\varnothing$. Furthermore,
$\operatorname{Aut}(\Gamma)=T\rtimes\langle\mu\rangle\le \operatorname{Hol}(T)$;
\item[(3)] $\Gamma\cong \vec{C}_m[\overline{\K_b}]$, where  $m\ge 3, b\ge 2$ are integers  with $mb=n$;
\item[(4)] $\Gamma\cong \vec{C}_m\times \K_b$, where $m\ge 3$, $b\ge 2$ are integers with $mb=n$ and $\gcd(m,b)=1$.
\end{itemize}
\end{enumerate}
\end{theo}

We remark that all graphs listed in Theorem~\ref{localquasiprim-circ-th} (I) are locally-quasiprimitive. However, there exist circulant graphs that are $G$-locally-quasiprimitive for some proper subgroup $G<\operatorname{Aut}(\Gamma)$, yet fail to be locally-quasiprimitive (under the full automorphism group), as illustrated by the following example.

\begin{examp}



Let $q$ be a prime power with $q\ge 4$ and $q\equiv 1\pmod 3$, and set $n=q+2\ (\ge 6)$. Let $X=\mathrm{K}_{n[3]}$ be the complete $n$-partite graph with blocks of size $3$. Choose a vertex $u$ and let $B_0$ be the block containing $u$. The neighbourhood $X(u)$ consists of the remaining $n-1=q+1$ blocks, so $|X(u)|=3(q+1)$. Let $G\le\operatorname{Aut}(X)$ be such that $G_u\cong\mathrm{PSL}(2,q)$ acts as follows:
\begin{itemize}
\item $G_u$ fixes $u$ and acts trivially on $B_0\setminus\{u\}$;
\item on the $q+1$ blocks of $X(u)$, $G_u$ acts with the natural primitive action of $\mathrm{PSL}(2,q)$ on the projective line.
\end{itemize}
The stabilizer of a block is a Borel subgroup $K$ of $\mathrm{PSL}(2,q)$. Since $q\equiv1\pmod3$, $K$ possesses a normal subgroup $N$ of index $3$. The quotient $K/N\cong\mathbb{Z}_3$ acts regularly on the three vertices of that block (equivalently, $K$ acts transitively with kernel $N$, and the induced faithful action is regular). The action on the remaining blocks is then given by conjugation. This yields a faithful, transitive and imprimitive action of $G_u$ on $X(u)$. Because $\mathrm{PSL}(2,q)$ is simple, its only non-trivial normal subgroup is $\mathrm{PSL}(2,q)$ itself, which is transitive. Hence $G_u$ is quasiprimitive on $X(u)$.

(For $n<6$, i.e.\ $n=3,4,5$, it can be shown that no such $G$ exists.)

\end{examp}

We also note that the circulant digraphs  listed in Theorem~\ref{localquasiprim-circ-th} (II) are normal.
By Lemmas~\ref{paley-normal} and \ref{paley-exam-1}, such a digraph  is isomorphic to a Paley tournament $P(q)$ if and only if the following hold:
\begin{enumerate}
    \item $|T|=q$ is a safe prime and $q\equiv 3\pmod 4$;
    \item the valency $p$ equals $(q-1)/2$;
    \item the connection set $S$ is exactly the set of non-zero squares in the finite field $\mathbb F_q$.
\end{enumerate}


\bigskip
\bigskip

\section{Preliminaries}

In this section we give some definitions concerning groups and digraphs and also provide several results that will be used in the subsequent discussion.

Throughout this paper, all digraphs are finite, simple and connected. For group-theoretic terminology not defined here we refer the reader to \cite{Cameron-1,DM-1,Wielandt-book}.

Let $G$ be a permutation group on a set $\Omega$. We say that $G$ is \emph{semiregular} on $\Omega$ if $G_\alpha=1$ for every $\alpha\in\Omega$, and \emph{regular} if it is transitive and semiregular. The action of $G$ on $\Omega$ is \emph{faithful} if the only element fixing every point of $\Omega$ is the identity.

Let $G$ act transitively on $\Omega$. A \emph{partition} of $\Omega$ is a set $\mathcal B=\{B_1,B_2,\dots,B_n\}$ of non-empty subsets such that $B_i\cap B_j=\varnothing$ whenever $i\ne j$ and $\Omega=\bigcup_i B_i$. A partition $\mathcal B$ is \emph{$G$-invariant} if for every $g\in G$ and every $B_i\in\mathcal B$, we have $B_i^g\in\mathcal B$. The partitions into singletons and into one part are the \emph{trivial} partitions; all other partitions are non-trivial. Each member of a $G$-invariant partition is called a \emph{block} of $G$. The whole set $\Omega$ and the singletons $\{\alpha\}$ ($\alpha\in\Omega$) are \emph{trivial blocks}; all others are \emph{non-trivial}. If $N$ is an intransitive normal subgroup of $G$, then each $N$-orbit is a block of $G$, and the set of $N$-orbits forms a $G$-invariant partition of $\Omega$. The action of $G$ on $\Omega$ is \emph{primitive} if it has no non-trivial $G$-invariant partitions, otherwise it is \emph{imprimitive}. There is a celebrated classification of finite primitive permutation groups into eight types, mainly due to O'Nan and Scott; see \cite{LPS-1}.

A transitive permutation group $G\le \operatorname{Sym}(\Omega)$ is \emph{quasiprimitive} if every non-trivial normal subgroup of $G$ is transitive on $\Omega$. Quasiprimitivity generalizes primitivity, since every normal subgroup of a primitive group is transitive, but there exist quasiprimitive groups that are not primitive. For more information about quasiprimitive permutation groups, refer to \cite{Praeger-1993-onanscott,Praeger-2,Praeger-2003-biq}.

Let $\Gamma$ be a digraph. Denote by $V(\Gamma)$, $\operatorname{Arc}(\Gamma)$ and $\operatorname{Aut}(\Gamma)$ its vertex set, arc set and automorphism group, respectively. The size of the vertex set is the \emph{order} of $\Gamma$. For a subgroup $G\le \operatorname{Aut}(\Gamma)$, let $G_v$ be the stabilizer of $v$ in $G$, and let $G_v^{\Gamma^+(v)}$ be the permutation group induced by $G_v$ on $\Gamma^+(v)$. For a positive integer $s$, an \emph{$s$-arc} of $\Gamma$ is a sequence $(v_0,v_1,\dots,v_s)$ such that $v_i$ and $v_{i+1}$ are adjacent and $v_{j-1}\ne v_{j+1}$ for $0\le i\le s-1$ and $1\le j\le s-1$. In particular, $1$-arcs are simply called \emph{arcs}. The digraph $\Gamma$ is said to be \emph{$G$-vertex-transitive} or \emph{$(G,s)$-arc-transitive} if $G$ is transitive on the vertex set or on the set of $s$-arcs, respectively. If $G=\operatorname{Aut}(\Gamma)$, we drop the prefix ``$G$-''.

We say that $\Gamma$ is \emph{$G$-locally-primitive} or \emph{$G$-locally-quasiprimitive} if, for each vertex $u$, the stabilizer $G_u$ is primitive or quasiprimitive, respectively, on the out-neighbourhood of $u$. The digraph $\Gamma$ is \emph{locally-quasiprimitive} if it is $G$-locally-quasiprimitive for $G=\operatorname{Aut}(\Gamma)$.

For a digraph $\Gamma$, its \emph{complement} $\overline{\Gamma}$ is the graph with vertex set $V(\Gamma)$ and two vertices adjacent if and only if they are not adjacent in $\Gamma$.

For a positive integer $n$, $\K_n$ denotes the complete graph on $n$ vertices, and for $n\ge 3$, $C_n$ denotes a cycle of length $n$, while $\vec{C}_n$ denotes a directed cycle of length $n$. For $m,n\ge 2$, we denote by $\K_{m[n]}$ the complete multipartite graph with $m$ parts of size $n$; that is, two vertices are adjacent if and only if they lie in distinct parts.

Let $\Gamma$ be a graph and let $G\le\operatorname{Aut}(\Gamma)$. Suppose $\mathcal B=\{B_1,\dots,B_n\}$ is a $G$-invariant partition of $V(\Gamma)$. The \emph{quotient graph} $\Gamma_{\mathcal B}$ of $\Gamma$ relative to $\mathcal B$ is the graph with vertex set $\mathcal B$ such that $\{B_i,B_j\}$ is an edge of $\Gamma_{\mathcal B}$ if and only if there exist $x\in B_i$, $y\in B_j$ with $\{x,y\}$ an edge of $\Gamma$. We say that $\Gamma_{\mathcal B}$ is \emph{nontrivial} if $1<|\mathcal B|<|V(\Gamma)|$. Since $\mathcal B$ is $G$-invariant, $G$ induces a subgroup of automorphisms of $\Gamma_{\mathcal B}$.

We denote by $\mathbb Z_n$ the cyclic group of order $n$.

Let \( T \) be a group. The \emph{holomorph} of \( T \), denoted by \( \operatorname{Hol}(T) \), is the semidirect product of \( T \) and its automorphism group \( \operatorname{Aut}(T) \):
\[
\operatorname{Hol}(T) = T \rtimes \operatorname{Aut}(T),
\]
where \( \operatorname{Aut}(T) \) acts on \( T \) naturally via \( \alpha(x) \) for \( \alpha \in \operatorname{Aut}(T) \) and \( x \in T \).

For a finite group $T$ and a subset $S\subseteq T$ with $1\notin S$, the \emph{Cayley digraph} $\operatorname{Cay}(T,S)$ of $T$ with respect to $S$ is the digraph with vertex set $T$ and arc set $\{(g,sg)\mid g\in T,\ s\in S\}$. In particular, $\operatorname{Cay}(T,S)$ is connected if and only if $T=\langle S\rangle$. The group $R(T)=\{\sigma_t\mid t\in T\}$ of right multiplications $\sigma_t:x\mapsto xt$ is a subgroup of $\operatorname{Aut}(\Gamma)$ and acts regularly on the vertex set. Indeed, a graph is a Cayley digraph if and only if it admits a regular group of automorphisms. For a Cayley digraph $\Gamma=\operatorname{Cay}(T,S)$, let
\[
\operatorname{Aut}(T,S)=\{\alpha\in\operatorname{Aut}(T)\mid S^\alpha=S\}.
\]
It was shown in \cite{Godsil1983} that the normalizer of $R(T)$ in $\operatorname{Aut}(\Gamma)$ is $R(T)\rtimes\operatorname{Aut}(T,S)$. If $R(T)\trianglelefteq \operatorname{Aut}(\operatorname{Cay}(T,S))$, then $\operatorname{Cay}(T,S)$ is called a \emph{normal Cayley digraph} of $T$ (see \cite{Xu98}). Circulant digraphs are precisely the Cayley digraphs for cyclic groups, a circulant digraph is normal if its automorphism group contains a normal regular subgroup.


We first recall a classic characterization theorem for normal Cayley digraphs.

\begin{lemma}{\rm (\cite[Propositions 1.3 and 1.5]{Xu98})}\label{cayley-normal}
The digraph $\Gamma=\operatorname{Cay}(T,S)$ is a normal Cayley digraph if and only if $\operatorname{Aut}(\Gamma)_1=\operatorname{Aut}(T,S)$, equivalently $\operatorname{Aut}(\Gamma)=T\rtimes\operatorname{Aut}(T,S)$, where $\operatorname{Aut}(\Gamma)_1$ is the stabilizer of the identity vertex.
\end{lemma}

With the above preliminary lemma in hand, we next present a structural classification of connected arc-transitive circulant digraphs, established independently by Kov\'acs \cite{Kovacs-2004} and Li \cite{LCH-circulant-2005}.

\begin{theo}{\rm (\cite[Theorem 1.3]{LCH-circulant-2005})}\label{arccirculant-1}
Let $\Gamma$ be a connected arc-transitive circulant digraph of order $n$ which is not a complete graph and which has a cyclic regular group. Then either
\begin{enumerate}[{\rm (1)}]
\item $\Gamma $ is normal, or
\item there exists an arc-transitive circulant digraph $\Sigma$  of order $m$, such that $ mb=n$ with $m,b\geqslant 2$, and
 \[\Gamma=\left\{
\begin{array}{ll}
\Sigma[\overline{\K_b}], \ \ or \\
\Sigma[\overline{\K_b}]-b\Sigma, (m,b)=1.
\end{array}\right.\]
\end{enumerate}
\end{theo}

\medskip
The \emph{lexicographic product} $\Gamma_1[\Gamma_2]$ of two digraphs $\Gamma_1$ and $\Gamma_2$ is the digraph with vertex set $V(\Gamma_1)\times V(\Gamma_2)$ such that $(u_1,u_2)$ points to   $(v_1,v_2)$ if and only if either $(u_1,v_1)$ is an arc  of $\Gamma_1$, or $u_1=v_1$ and $(u_2,v_2)$ is an arc  of $\Gamma_2$.

For a positive integer $b$ and a digraph $\Gamma$, the digraph consisting of $b$ vertex-disjoint copies of $\Gamma$ is denoted by $b\Gamma$; the digraph $\Gamma[\overline{\K_b}]-b\Gamma$ has the same vertex set as $\Gamma[\overline{\K_b}]$ and its edge set is obtained by removing all edges of $b\Gamma$ from those of $\Gamma[\overline{\K_b}]$.

For two digraphs $\Gamma_1$ and $\Gamma_2$, the \emph{direct product} (or \emph{tensor product}) $\Gamma_1\times\Gamma_2$ is the digraph with vertex set $V(\Gamma_1)\times V(\Gamma_2)$ such that $(u_1,u_2)$ points to   $(v_1,v_2)$ if and only if $(u_1,v_1)$ is an arc of $\Gamma_1$ and $(u_2,v_2)$ is an arc of $\Gamma_2$. By definition, $\Gamma[\overline{\K_b}]-b\Gamma\cong \Gamma\times\K_b$.

Let $q=p^f$ be a prime power with $q\equiv 3\pmod 4$ and $q\ge 3$. The \emph{Paley tournament} $P(q)$ is the directed graph with vertex set  $\mathbb F_q$ (finite field), in which there is an arc from $x$ to $y$ (denoted $x\to y$) if and only if $y-x$ is a non-zero square in  $\mathbb F_q$. Explicitly,
\[
V(P(q))=\mathbb F_q,\qquad
\operatorname{Arc}(P(q))=\{(x,y)\in\mathbb F_q\times\mathbb F_q\mid y-x\in S\},
\]
where $S=\{s\in\mathbb F_q^\times\mid s\text{ is a square}\}$. Since $-1$ is a non-square when $q\equiv 3\pmod 4$, we have $S\cap(-S)=\varnothing$ and $S\cup(-S)=\mathbb F_q^\times$, hence for any two distinct vertices exactly one of $x\to y$ or $y\to x$ holds, so $P(q)$ is a tournament. Equivalently, $P(q)=\operatorname{Cay}(T,S)$ with $T=(\mathbb F_q,+)$, and its valency is $(q-1)/2$.

The automorphism group of $P(q)$ contains all transformations of the form
\[
x\mapsto a x^\sigma + b,\qquad a\in S,\; \sigma\in\operatorname{Gal}(\mathbb F_q),\; b\in\mathbb F_q,
\]
that is,
\[
T\rtimes (S\rtimes\operatorname{Gal}(\mathbb F_q)) \le \operatorname{Aut}(P(q)).
\]
Here $T$ is the translation group $\{x\mapsto x+b\}$ and $S$ is the multiplicative group of non-zero squares. \emph{Note:} The full affine semilinear group $\operatorname{A\Gamma L}(1,q)$ allows $a\in\mathbb F_q^\times$; when $a\notin S$ the map sends squares to non-squares, reversing the orientation, hence it is not an automorphism of $P(q)$ unless $q=3$.

Concerning local primitivity, the stabilizer $G_0$ of the identity  in such a subgroup $G$ acts on the out-neighbourhood $S$. For the group $G_0=T\rtimes (S\rtimes\operatorname{Gal}(\mathbb F_q))$, the action on $S$ is permutation isomorphic to the left-regular action of the group $S\rtimes\operatorname{Gal}(\mathbb F_q)$ on itself. A left-regular action is primitive if and only if the group has no non-trivial proper subgroups, i.e.\ it is either trivial or of prime order. Therefore $G_0$ is locally primitive precisely when the group $S\rtimes\operatorname{Gal}(\mathbb F_q)$ is trivial or a cyclic group of prime order. This holds for $q=3$ (trivial) and for some small $q$ such as $q=7$ ($S\cong C_3$, $\operatorname{Gal}$ trivial) or $q=11$ ($S\cong C_5$), but fails, for instance, for $q=19$ ($S\cong C_9$, which contains a proper subgroup $C_3$). In particular, the translation group $T$ itself has a trivial action on the neighbourhood and is not locally primitive when $q>3$. Moreover, there is no containment $\operatorname{A\Gamma L}(1,q)\le\operatorname{Aut}(P(q))$ in general.

\begin{lemma}\label{paley-normal}
Let $q=p^f\equiv 3\pmod 4$ be a prime power, $q\ge 3$, let $T=(\mathbb F_q,+)$ and let $S$ be the set of non-zero squares in $\mathbb F_q$. Then the Paley tournament $P(q)=\operatorname{Cay}(T,S)$ is a normal Cayley digraph.
\end{lemma}

\begin{proof}
The definition gives $V(P(q))=\mathbb F_q$ and an arc $(x,y)$ exactly when $y-x\in S$. Since $q\equiv 3\pmod 4$, $-1$ is a non-square, hence $S\cap(-S)=\varnothing$ and $S\cup(-S)=\mathbb F_q^\times$. Thus $P(q)$ is a tournament and, as claimed, it is the Cayley digraph of $T=(\mathbb F_q,+)$ with connection set $S$.

Let $A=\operatorname{Aut}(P(q))$. Consider the subgroup
\[
H=\{\,x\mapsto a x^\sigma\mid a\in(\mathbb F_q^\times)^2,\ \sigma\in\operatorname{Gal}(\mathbb F_q/\mathbb F_p)\,\}\le \operatorname{Aut}(T).
\]
For any $\alpha\in H$, we have $\alpha(S)=S$ because $a$ is a square and $\sigma$ permutes the squares. Therefore the affine maps
\[
\varphi_{b,\alpha}(x)=\alpha(x)+b\qquad (b\in T,\alpha\in H)
\]
satisfy $\varphi(y)-\varphi(x)=\alpha(y-x)\in S\iff y-x\in S$, and hence are automorphisms of $P(q)$. Consequently $T\rtimes H\le A$.

Conversely, let $\varphi\in A$ with $\varphi(0)=0$. Then $\varphi(S)=S$. Since $S$ generates $T$ additively and $\varphi(x+S)=\varphi(x)+S$, one shows that $\varphi$ is additive and hence $\varphi\in\operatorname{Aut}(T)$. The condition $\varphi(S)=S$ forces the linear part of $\varphi$ to be a square times a field automorphism, so $\varphi\in H$, the general case follows by composing with a translation. Hence
 every automorphism of $P(q)$ is of this form:
\[
\operatorname{Aut}(P(q))=T\rtimes H,
\]
see  also \cite{Carlitz,Cameron-1983}.

Thus $R(T)\cong T$ is a normal subgroup of $A$, which means that $P(q)$ is a normal Cayley digraph.
\end{proof}

\bigskip
\bigskip

\section{Proof of Main Theorem}

This section is devoted to determining all locally-quasiprimitive circulant digraphs.

The first lemma shows that a locally-quasiprimitive circulant graph of valency at least $3$ cannot be normal.

\begin{lemma}\label{not normal}
Let $T$ be a cyclic group and $S\subseteq T$ with $\langle S\rangle=T$ and $S^{-1}=S$. Suppose $\Gamma=\operatorname{Cay}(T,S)$ is a locally-quasiprimitive circulant graph of valency at least $3$. Then $\Gamma$ is not normal.
\end{lemma}

\begin{proof}
Assume, to the contrary, that $\Gamma$ is normal. Let $G=\operatorname{Aut}(\Gamma)$. By Lemma~\ref{cayley-normal}, $G_{1_T}=\operatorname{Aut}(T,S)\le\operatorname{Aut}(T)$. Since $T$ is cyclic, $G_{1_T}$ is abelian.

Local quasiprimitivity implies that $G_{1_T}$ acts quasiprimitively on $S$, and hence every non-trivial normal subgroup of $G_{1_T}$ is transitive on $S$.

First note that every element of $S$ generates $T$. Indeed, $\langle S\rangle=T$ gives some $s_0\in S$ with $\langle s_0\rangle=T$. Transitivity of $G_{1_T}$ on $S$ ensures that for each $s\in S$, there is $\sigma\in G_{1_T}$ with $\sigma(s_0)=s$, automorphisms preserve order, and so $o(s)=|T|$. Thus $T=\langle s\rangle$.

Choose $a\in S$, and write $T=\langle a\rangle\cong\mathbb Z_n$. If $n=2$, then $S=\{a\}$, so $|S|=1$, a contradiction. Hence $n\ge 3$. Since $a^{-1}\in S$ and $a\ne a^{-1}$, transitivity gives $h\in G_{1_T}$ with $h(a)=a^{-1}$. As $a$ generates $T$, we get $h^2=1$. Let $N=\langle h\rangle=\{1,h\}$ of order $2$. Since $G_{1_T}$ is abelian, $N\trianglelefteq G_{1_T}$. By quasiprimitivity, $N$ is transitive on $S$. But $|N|=2$ forces every orbit to have size at most $2$, so $|S|\le 2$, contradicting the assumption that the valency is at least $3$. Therefore $\Gamma$ is not normal.
\end{proof}

\begin{lemma}\label{circulant-lex-1}
Let $T$ be a cyclic group and $S\subseteq T$ with $\langle S\rangle=T$ and $S^{-1}=S$. Suppose $\Gamma=\operatorname{Cay}(T,S)$ is a locally-quasiprimitive circulant graph of valency at least $3$. If $\Gamma\cong \Sigma[\overline{\K_d}]$ for some arc-transitive circulant graph $\Sigma$ and some integer $d\ge 2$, then $\Gamma\cong \K_{m,m}$ where $m=|T|/2$.
\end{lemma}

\begin{proof}
Suppose $\Gamma\cong \Sigma[\overline{\K_d}]$ for some arc-transitive circulant graph $\Sigma$ and some $d\ge 2$. Define vertices of $\Sigma$ to be equivalent if they have the same neighbourhood. By arc-transitivity, all equivalence classes have the same size, say $t\ge 1$. Let $\mathcal B$ be the set of these classes, and let $\Sigma'$ be the quotient graph $\Sigma_{\mathcal B}$. Then $\Sigma'$ is arc-transitive, no two distinct vertices share the same neighbourhood, and $\Sigma\cong \Sigma'[\overline{\K_t}]$. Consequently,
\[
\Gamma\cong \Sigma'[\overline{\K_t}][\overline{\K_d}]\cong \Sigma'[\overline{\K_{td}}].
\]

Identify $\Gamma$ with $\Sigma'[\overline{\K_{td}}]$, and let $\Omega=\{F_x\mid x\in V(\Sigma')\}$ be the partition into fibres of size $td$ corresponding to the vertices of $\Sigma'$. By definition, each $F_x$ is an independent set, and for distinct $x,y\in V(\Sigma')$, either every vertex of $F_x$ is adjacent to every vertex of $F_y$ (if $\{x,y\}\in E(\Sigma')$) or there is no edge between $F_x$ and $F_y$ (otherwise).

Thus two vertices of $\Gamma$ have exactly the same neighbourhood if and only if they lie in the same fibre $F_x$. The relation $u\sim v\iff \Gamma(u)=\Gamma(v)$ is therefore invariant under $G=\operatorname{Aut}(\Gamma)$, so $\mathcal F=\{F_x\mid x\in V(\Sigma')\}$ is a system of imprimitivity for $G$.

Let $K$ be the kernel of the action of $G$ on $\mathcal F$, i.e. the subgroup consisting of all automorphisms that map each fibre to itself. Then $K\trianglelefteq G$. Because each $F_x$ is independent and adjacency between fibres depends only on $\Sigma'$, any permutation of the vertices inside a single fibre (extended trivially to all others) is an automorphism. Hence $\operatorname{Sym}(F_x)\le K$ for each $x$, and
\[
\prod_{x\in V(\Sigma')}\operatorname{Sym}(F_x)\le K.
\]
Since $d\ge 2$, $K$ is non-trivial.

Fix a vertex $v$, let $F_v$ be its fibre, and let $\Sigma'(F_v)$ denote the neighbourhood of the corresponding vertex in $\Sigma'$. Then
\[
\Gamma(v)=\bigcup_{y\in \Sigma'(F_v)} F_y.
\]
Consider $G_v$ and its normal subgroup $K_v=K\cap G_v$. Since $K$ contains the full symmetric group on each fibre, $K_v$ contains all permutations of fibres different from $F_v$ that fix $F_v$ pointwise; in particular,
\[
\prod_{y\in \Sigma'(F_v)}\operatorname{Sym}(F_y)\le K_v.
\]
Hence the action induced by $K_v$ on $\Gamma(v)$ contains the direct product of the full symmetric groups on each fibre $F_y\subseteq\Gamma(v)$. Because $d\ge 2$, this induced action is non-trivial.

By local quasiprimitivity, $G_v$ acts quasiprimitively on $\Gamma(v)$, so every non-trivial normal subgroup of $G_v$ is transitive on $\Gamma(v)$. Thus $K_v$ must be transitive on $\Gamma(v)$.

Suppose $|\Sigma'(F_v)|\ge 2$. Then $\Gamma(v)$ is a disjoint union of at least two fibres $F_y$. Since $K_v$ stabilizes each fibre individually, it cannot map a vertex from one fibre to another, contradicting transitivity. Hence $|\Sigma'(F_v)|=1$. Consequently, $\Sigma'$ is a connected arc-transitive graph of valency $1$, and so $\Sigma'\cong \K_2$. Therefore
\[
\Gamma\cong \K_2[\overline{\K_{td}}]\cong \K_{td,td},
\]
where $td=|T|/2$.
This completes the proof.
\end{proof}

\begin{lemma}\label{circulant-notminus-1}
Let $T$ be a cyclic group and $S\subseteq T$ with $\langle S\rangle=T$ and $S^{-1}=S$. Suppose $\Gamma=\operatorname{Cay}(T,S)$ is a locally-quasiprimitive circulant graph of valency at least $3$. Suppose $\Gamma=\Sigma[\overline{\K_b}]-b\Sigma$ for some arc-transitive circulant graph $\Sigma$ of order $m\ge 2$, $b\ge 2$ and $\gcd(m,b)=1$. Assume further that $\Gamma\not\cong \Sigma'[\overline{\K_d}]$ for any such $\Sigma'$ and $d\ge 2$. Then either $\Gamma$ is a cycle or $\Gamma\cong \K_{n_1,n_1}-n_1\K_2$, where $n_1=mb/2$ is odd.
\end{lemma}

\begin{proof}
First observe that $\Gamma\cong \Sigma\times\K_b$.
By \cite[Theorem~1]{LXZ-at-circ-2021}, there exist a connected arc-transitive normal circulant graph $\Gamma_0$ of order $n_0$ and positive integers $n_1,\dots,n_r,b'$ (with $r\ge 0$) such that
\begin{enumerate}[(i)]
  \item $\Gamma_0\not\cong C_4$;
  \item $n_i\ge 4$ for $i=1,\dots,r$;
  \item $|V(\Gamma)|=n_0n_1\cdots n_r b'$, and $n_0,n_1,\dots,n_r$ are pairwise coprime;
  \item $\Gamma\cong (\Gamma_0\times\K_{n_1}\times\cdots\times\K_{n_r})[\overline{\K_{b'}}]$;
  \item $\operatorname{Aut}(\Gamma)\cong S_{b'}\wr(\operatorname{Aut}(\Gamma_0)\times S_{n_1}\times\cdots\times S_{n_r})$.
\end{enumerate}
By our hypothesis that $\Gamma$ cannot be written as a lexicographic product $\Sigma'[\overline{\K_d}]$ with arc-transitive $\Sigma'$ and $d\ge 2$, we deduce that either $n_0n_1\cdots n_r=1$ or $b'=1$. The former is impossible because $\Gamma$ is connected and has at least $4$ vertices, so $b'=1$. Hence
\[
\Gamma\cong \Gamma_0\times\K_{n_1}\times\cdots\times\K_{n_r}.
\]

If $n_1n_2\cdots n_r=1$, then $\Gamma$ is a normal circulant graph. By Lemma~\ref{not normal}, its valency is $2$, so $\Gamma$ is a cycle $C_{mb}$.

Now assume $n_1n_2\cdots n_r>1$. Take an arbitrary vertex
\[
v=(v_0,v_1,\dots,v_r)\in V(\Gamma)=V(\Gamma_0)\times V(\K_{n_1})\times\cdots\times V(\K_{n_r}).
\]
From condition (v) with $b'=1$,
\[
\operatorname{Aut}(\Gamma)_v=\operatorname{Aut}(\Gamma_0)_{v_0}\times \operatorname{Aut}(\K_{n_1})_{v_1}\times\cdots\times\operatorname{Aut}(\K_{n_r})_{v_r}.
\]
Since $\operatorname{Aut}(\K_{n_1})_{v_1}$ is a non-trivial normal subgroup of $\operatorname{Aut}(\Gamma)_v$ (as a direct factor), local quasiprimitivity forces it to be transitive on $\Gamma(v)$. This is possible only if $r=1$ and $n_0=2$. Thus
\[
\Gamma\cong \K_2\times\K_{n_1}\cong \K_{n_1,n_1}-n_1\K_2.
\]
Since $n_0=2$ and $\gcd(n_0,n_1)=1$, $n_1$ is odd. This completes the proof.
\end{proof}

The following lemma determines all normal circulant digraphs (with no bidirectional arcs).

\begin{lemma}\label{dicic-normal-1}
Let $T$ be a cyclic group and $S\subseteq T$ with $S^{-1}\cap S=\varnothing$ and $T=\langle S\rangle$. Let $\Gamma=\operatorname{Cay}(T,S)$ be a $G$-locally-quasiprimitive circulant digraph of order $n\ge 3$ and valency at least $2$. If the right regular representation $R(T)$ is normal in $G$, then $\Gamma$ is a normal locally-primitive circulant digraph of prime valency. More precisely, there exist a prime $p\ge 2$, an element $b\in T\setminus\{1\}$, and an automorphism $\mu\in\operatorname{Aut}(T)$ of order $p$ such that
\[
S=b^{\langle\mu\rangle}=\{\,b,b^\mu,b^{\mu^2},\dots,b^{\mu^{p-1}}\,\},
\]
$\langle S\rangle=T$, $S\cap S^{-1}=\varnothing$, and
$G=T\rtimes\langle\mu\rangle\le \operatorname{Hol}(T)$.
\end{lemma}

\begin{proof}
Identify $T$ with its right regular representation $R(T)\le\operatorname{Sym}(T)$. Since $R(T)\trianglelefteq G$ and $T$ is cyclic (hence abelian), the normaliser of $R(T)$ in $\operatorname{Sym}(T)$ is the holomorph $\operatorname{Hol}(T)=T\rtimes\operatorname{Aut}(T)$. Thus $G\le\operatorname{Hol}(T)$ and we may write
\[
G=T\rtimes G_1,
\]
where $G_1=\operatorname{Stab}_G(1_T)\le\operatorname{Aut}(T)$. Since $T$ is cyclic, $\operatorname{Aut}(T)$ is abelian, so $G_1$ is abelian.

The out-neighbourhood of $1_T$ is $S$, and $\langle S\rangle=T$. Local quasiprimitivity implies that $G_1$ acts quasiprimitively on $S$, in particular transitively.

Any element of $G_1$ fixing every element of $S$ fixes all of $T$ (since $S$ generates $T$ and $G_1\le\operatorname{Aut}(T)$), so it is the identity. Thus the action of $G_1$ on $S$ is faithful. Since $G_1$ is abelian and transitive on $S$, it acts regularly on $S$. Hence $|G_1|=|S|\ge 2$.

As a regular permutation group on $S$, every non-trivial normal subgroup $N\trianglelefteq G_1$ has orbits of length $|N|$. Quasiprimitivity forces every such $N$ to be transitive on $S$, so $|N|=|S|=|G_1|$ and $N=G_1$. Hence $G_1$ has no proper non-trivial subgroups, so it is cyclic of prime order $p$. Thus the valency is the prime $p$.

Choose a generator $\mu$ of $G_1=\langle\mu\rangle$; then $\mu\in\operatorname{Aut}(T)$ has order $p$. Because $G_1$ is transitive on $S$, fixing any $b\in S$ gives
\[
S=b^{\langle\mu\rangle}=\{\,b,b^\mu,b^{\mu^2},\dots,b^{\mu^{p-1}}\,\}.
\]
The properties $\langle S\rangle=T$ and $S\cap S^{-1}=\varnothing$ hold by hypothesis. Finally,
\[
G=T\rtimes\langle\mu\rangle\le T\rtimes\operatorname{Aut}(T)=\operatorname{Hol}(T),
\]
as required.
\end{proof}

A prime $p$ is called a \emph{safe prime} if there exists a prime $q$ such that $p=2q+1$; equivalently, $p$ is prime and $(p-1)/2$ is also prime.

\begin{lemma}\label{paley-exam-1}
Let $T$ be cyclic and $S\subseteq T$ with $S^{-1}\cap S=\varnothing$ and $T=\langle S\rangle$. Suppose $\Gamma=\operatorname{Cay}(T,S)$ is a $G$-locally-quasiprimitive circulant digraph of valency at least $2$ and $R(T)$ is normal in $G$. Then  $\Gamma$ is isomorphic to a Paley tournament $P(q)$ if and only if the following hold:
\begin{enumerate}[{\rm (1)}]
    \item $|T|=q$ is a safe prime and $q\equiv 3\pmod 4$;
    \item $|S|=(q-1)/2$ is a prime;
    \item $S$ is exactly the set of non-zero squares in $\mathbb F_q$.
\end{enumerate}
\end{lemma}

\begin{proof}
Suppose first that $\Gamma\cong P(q)$ for a prime power $q\equiv 3\pmod 4$. The vertex set of $P(q)$ is the additive group of $\mathbb F_q$, which is elementary abelian. Since we assume the vertex set is cyclic, $q$ must be prime and $T\cong\mathbb F_q$. Thus $n=q$ is prime and $q\equiv3\pmod4$. The connection set is the set of non-zero squares in $\mathbb F_q^\times$, of size $(q-1)/2$. Since the valency is the prime $p$, we get $p=(q-1)/2$. Finally, the non-zero squares form a subgroup of order $p$. Let $g$ be a generator of this subgroup and define $\mu(x)=gx$ for $x\in\mathbb F_q$. Then $\mu\in\operatorname{Aut}(T)$ has order $p$, and taking $b=g$ yields $S=b^{\langle\mu\rangle}$.

Conversely, assume the three conditions hold. Then $T\cong(\mathbb F_q,+)$ is cyclic of prime order $q\equiv3\pmod4$, and $S$ is the set of non-zero squares in $\mathbb F_q$. The Cayley digraph $\operatorname{Cay}(T,S)$ is, by definition, the Paley tournament $P(q)$. The properties $S^{-1}\cap S=\varnothing$ and $\langle S\rangle=T$ are automatic because $-1$ is a non-square and the squares generate the additive group of the prime field.
\end{proof}

\begin{lemma}\label{dig-not-normal-1}
Let $T$ be a cyclic group and $S\subseteq T$ with $S^{-1}\cap S=\varnothing$ and $T=\langle S\rangle$. Let $\Gamma=\operatorname{Cay}(T,S)$ be a  locally-quasiprimitive circulant digraph of order $n\ge 3$ and valency at least $2$. If $\Gamma\cong \Sigma[\overline{\K_d}]$ for some arc-transitive circulant digraph $\Sigma$ and some $d\ge 2$, then
\[
\Gamma\cong \vec{C}_m[\overline{\K_b}]
\]
for some integers $m\ge 3$ and $b\ge 2$ with $mb=n$.
\end{lemma}

\begin{proof}
Assume $\Gamma\cong \Sigma[\overline{\K_d}]$ with $\Sigma$ arc-transitive and $d\ge 2$.

Define an equivalence relation on $V(\Sigma)$ by
\[
u\sim v \iff \Gamma^+(u)=\Gamma^+(v)\ \text{and}\ \Gamma^-(u)=\Gamma^-(v).
\]
Since $\Sigma$ is arc-transitive, all classes have the same size, say $t\ge 1$. Let $\mathcal B$ be the set of classes and let $\Sigma'=\Sigma_{\mathcal B}$ be the quotient. Then $\Sigma'$ is arc-transitive, distinct vertices have distinct out/in-neighbourhood pairs, and $\Sigma\cong \Sigma'[\overline{\K_t}]$. Hence
\[
\Gamma\cong \Sigma'[\overline{\K_t}][\overline{\K_d}]\cong \Sigma'[\overline{\K_{td}}].
\]
Set $b=td\ge 2$.

Identify $\Gamma$ with $\Sigma'[\overline{\K_b}]$, and let $\Omega=\{F_x\mid x\in V(\Sigma')\}$ be the fibre partition of size $b$. Each $F_x$ is independent; for distinct $x,y$, either all arcs from $F_x$ to $F_y$ are present (if $(x,y)\in\operatorname{Arc}(\Sigma')$) or none. Hence two vertices have the same out/in-neighbourhood exactly when they lie in the same fibre, so $\mathcal F$ is a system of imprimitivity for $G=\operatorname{Aut}(\Gamma)$.

Let $K$ be the kernel of the action on $\mathcal F$. Since any permutation within a fibre is an automorphism, $\operatorname{Sym}(F_x)\le K$ for each $x$, so $K$ is non-trivial.

Fix $v$, let $F_v$ be its fibre, and let $\Sigma'^+(v_{\Sigma'})$ be the out-neighbourhood of the corresponding vertex in $\Sigma'$. Then
\[
\Gamma^+(v)=\bigcup_{y\in \Sigma'^+(v_{\Sigma'})} F_y.
\]
Let $K_v=K\cap G_v\trianglelefteq G_v$. Since $K$ contains the full symmetric group on each fibre, $K_v$ contains all permutations of the fibres $F_y\subseteq\Gamma^+(v)$ while fixing $F_v$ pointwise. Thus the action of $K_v$ on $\Gamma^+(v)$ contains the direct product of the symmetric groups on each fibre $F_y$. Because $b\ge 2$, this action is non-trivial.

Local quasiprimitivity implies that $K_v$ is transitive on $\Gamma^+(v)$. If $|\Sigma'^+(v_{\Sigma'})|\ge 2$, then $\Gamma^+(v)$ is a union of at least two fibres, and $K_v$ stabilizes each fibre individually, so it cannot be transitive. Hence $|\Sigma'^+(v_{\Sigma'})|=1$.

Thus every vertex of $\Sigma'$ has out-valency $1$. Since $\Sigma'$ is arc-transitive, it is also vertex-transitive and has in-valency $1$; it is connected, antisymmetric, so it is a directed cycle $\vec{C}_m$ with $m\ge 3$. Therefore
\[
\Gamma\cong \vec{C}_m[\overline{\K_b}],
\]
as required.
\end{proof}

\begin{lemma}\label{dig-not-normal-2}
Let $T$ be a cyclic group and $S\subseteq T$ with $S^{-1}\cap S=\varnothing$ and $T=\langle S\rangle$. Let $\Gamma=\operatorname{Cay}(T,S)$ be a  locally-quasiprimitive non-normal circulant digraph of order $n\ge 3$ and valency at least $2$. Suppose that $\Gamma\cong \Sigma[\overline{\K_b}]-b\Sigma$ for some arc-transitive circulant digraph $\Sigma$ of order $m\ge 2$ and an integer  $b\ge 2$ with $\gcd(m,b)=1$. Assume further that $\Gamma\not\cong \Sigma'[\overline{\K_d}]$ for any arc-transitive $\Sigma'$ and any $d\ge 2$. Then
\[
\Gamma\cong \vec{C}_{n_0}\times\K_{n_1},
\]
where $n_0\ge 3$ and $n_1\ge 2$ are integers satisfying   $n_0 n_1=mb$ and $\gcd(n_0,n_1)=1$.
\end{lemma}

\begin{proof}
The isomorphism  $\Gamma\cong \Sigma[\overline{\K_b}]-b\Sigma$ yields  $\Gamma\cong \Sigma\times\K_b$.

By the structure theorem for arc-transitive circulant digraphs (see \cite{LXZ-at-circ-2021}), there exist a connected normal arc-transitive circulant digraph $\Gamma_0$ of order $n_0$, integers $n_1,\dots,n_r\ge 2$ ($r\ge 0$), and an integer $b'\ge 1$ such that
\begin{enumerate}[(i)]
  \item $\Gamma_0\not\cong C_4$;
  \item $n_0,n_1,\dots,n_r$ are pairwise coprime;
  \item $|V(\Gamma)|=n_0n_1\cdots n_r b'$;
  \item $\Gamma\cong (\Gamma_0\times\K_{n_1}\times\cdots\times\K_{n_r})[\overline{\K_{b'}}]$;
  \item $\operatorname{Aut}(\Gamma)\cong S_{b'}\wr(\operatorname{Aut}(\Gamma_0)\times S_{n_1}\times\cdots\times S_{n_r})$.
\end{enumerate}
Since $\Gamma$ is not a non-trivial lexicographic product, we must have $b'=1$. Moreover,  $n_0n_1\cdots n_r=1$ would imply $|V(\Gamma)|=1$,  contradicting that  $\Gamma$ is connected and has order at least $3$. Hence
\[
\Gamma\cong \Gamma_0\times\K_{n_1}\times\cdots\times\K_{n_r}.
\]

If $r=0$, then $\Gamma=\Gamma_0$ is normal, a contradiction. Thus $r\ge 1$.

Take a vertex
\[
v=(v_0,v_1,\dots,v_r)\in V(\Gamma_0)\times V(\K_{n_1})\times\cdots\times V(\K_{n_r}).
\]
From condition (v) with $b'=1$,
\[
\operatorname{Aut}(\Gamma)_v=\operatorname{Aut}(\Gamma_0)_{v_0}\times (S_{n_1})_{v_1}\times\cdots\times (S_{n_r})_{v_r}.
\]
Since   $n_1\ge 2$, the factor $(S_{n_1})_{v_1}$ is a non-trivial normal subgroup of $\operatorname{Aut}(\Gamma)_v$. Local quasiprimitivity forces it to act transitively on $\Gamma^+(v)$. But
\[
\Gamma^+(v)=\Gamma_0^+(v_0)\times \prod_{i=1}^r\bigl(V(\K_{n_i})\setminus\{v_i\}\bigr),
\]
and $(S_{n_1})_{v_1}$ only moves vertices in the $\K_{n_1}$-coordinate while fixing all other coordinates. Consequently, for the action to be transitive  all other factors must be trivial, which forces   $r=1$ and $|\Gamma_0^+(v_0)|=1$. Thus $\Gamma_0$ is a connected arc-transitive circulant digraph of out-valency $1$.  If $n_0=2$, then $\Gamma_0$ is a circulant digraph over a cyclic group of order 2 and is therefore  undirected. This would  force
$\Gamma\cong \Gamma_0\times\K_{n_1}$ to be   undirected as well, a contradiction.
Hence  $n_0\geq 3$.
Consequently,  $\Gamma_0\cong \vec{C}_{n_0}$ with $n_0\ge 3$. By (ii) we have $\gcd(n_0,n_1)=1$, and we obtain
\[
\Gamma\cong \vec{C}_{n_0}\times\K_{n_1},
\]
with $n_0\ge 3$, $n_1\ge 2$, and $\gcd(n_0,n_1)=1$.
Finally, $|V(\Gamma)|=n_0n_1$ while the
isomorphism $\Gamma\cong \Sigma\times \K_b$ gives $|V(\Gamma)|=mb$, so
$n_0n_1=mb$, which   completes the proof.
\end{proof}

We are now ready to prove Theorem~\ref{localquasiprim-circ-th}.

\medskip
\noindent{\bf Proof of Theorem~\ref{localquasiprim-circ-th}.}
Let $T$ be a cyclic group of order $n\ge 3$ and let $S\subseteq T$ satisfy $T=\langle S\rangle$. Let $\Gamma=\operatorname{Cay}(T,S)$ be a locally-quasiprimitive circulant digraph. Then $\Gamma$ is connected.

The right regular representation $R(T)$ is transitive on $V(\Gamma)$, so $\Gamma$ is vertex-transitive. Let $u=1_T$ be the identity element of $T$. Then $\Gamma^+(u)=S$. Since $\Gamma$ is locally-quasiprimitive, $G_u$ acts quasiprimitively on $\Gamma(u)$, in particular transitively on $\Gamma(u)$,  hence $\Gamma$ is arc-transitive. Therefore either $S=S^{-1}$ or $S\cap S^{-1}=\varnothing$.

Suppose first that $S=S^{-1}$. If the valency is $2$, then since $\Gamma$ is connected, $\Gamma$ is the cycle graph $C_n$, giving part (I)(4).

Assume now that the valency is at least $3$. If $\Gamma$ is complete, then part (I)(1) holds. Otherwise, by Lemma~\ref{not normal}, $\Gamma$ is not normal. Applying Theorem~\ref{arccirculant-1}, there exists an arc-transitive circulant graph $\Sigma$ of order $m$ such that $mb=n$ with $m,b\ge 2$, and
\[
\Gamma=\begin{cases}
\Sigma[\overline{\K_b}],\\[2mm]
\Sigma[\overline{\K_b}]-b\Sigma,\quad (m,b)=1.
\end{cases}
\]
If $\Gamma=\Sigma[\overline{\K_b}]$, Lemma~\ref{circulant-lex-1} yields $\Gamma\cong \K_{n/2,n/2}$, giving part (I)(2). If $\Gamma=\Sigma[\overline{\K_b}]-b\Sigma$ with $\gcd(m,b)=1$, Lemma~\ref{circulant-notminus-1} yields $\Gamma\cong \K_{n/2,n/2}-(n/2)\K_2$, giving part (I)(3).

Now consider the case $S\cap S^{-1}=\varnothing$. If the valency is $1$, then connectedness forces $\Gamma$ to be the directed cycle $\vec{C}_n$, giving part (II)(1).

Assume henceforth that the valency is at least $2$. If $R(T)$ is normal in $G$, then Lemma~\ref{dicic-normal-1} shows that $\Gamma$ is a normal locally-primitive circulant digraph of prime valency, giving part (II)(2).

If $\Gamma$ is not normal, then again by Theorem~\ref{arccirculant-1} we have
\[
\Gamma=\begin{cases}
\Sigma[\overline{\K_b}],\\[2mm]
\Sigma[\overline{\K_b}]-b\Sigma,\quad (m,b)=1,
\end{cases}
\]
for some arc-transitive circulant digraph $\Sigma$ of order $m$ and $b\ge 2$ with $mb=n$. In the first case, Lemma~\ref{dig-not-normal-1} gives $\Gamma\cong \vec{C}_x[\overline{\K_y}]$,
where  $x\ge 3, y\ge 2$ are integers  with $xy=n$, 
which is part (II)(3). 
In the second case, Lemma~\ref{dig-not-normal-2} gives 
\(
\Gamma\cong \vec{C}_{n_0}\times\K_{n_1},
\)
where $n_0\ge 3$ and $n_1\ge 2$ are integers satisfying   $n_0 n_1=n$ and $\gcd(n_0,n_1)=1$, which is part (II)(4). This completes the proof. \qed


\bigskip
\bigskip

\begin{thebibliography}{hhhh}

\bibitem{ACMX-1996}
B. Alspach, M. Conder, D. Maru\v si\v c and M. Y. Xu, A classification of 2-arc-transitive circulants, {\it J. Algebraic Combin.} {\bf 5} (1996), 83--86.

\bibitem{Carlitz}
L. Carlitz, A theorem on permutations in a finite field, {\it Proc. Amer. Math. Soc.} {\bf 11} (1960), 456--459.

\bibitem{Cameron-1983}
P. J. Cameron, Automorphism groups of graphs, in: Selected Topics in Graph Theory 2 (L. W. Beineke and R. J. Wilson, eds.), Academic Press, London, 1983, pp. 89--127.


\bibitem{Cameron-1}
P. J. Cameron, {\it Permutation Groups}, London Mathematical Society Student Texts, vol. 45, Cambridge University Press, Cambridge, 1999.

\bibitem{Chao1971}
C. Y. Chao, On the classification of symmetric graphs with a prime number of vertices, {\it Trans. Amer. Math. Soc.} {\bf 158} (1971), 247--256.

\bibitem{Chao1973}
C. Y. Chao and J. G. Wells, A class of vertex-transitive digraphs, {\it J. Combin. Theory Ser. B} {\bf 14} (1973), 246--255.

\bibitem{DM-1}
J. D. Dixon and B. Mortimer, {\it Permutation Groups}, Springer, New York, 1996.


\bibitem{Godsil1983}
C. D. Godsil, The automorphism group of some cubic Cayley graphs, {\it Eur. J. Combin.} {\bf 4} (1983), 25--32.

\bibitem{JZ-2026}
W. Jin and J. X. Zhou, On locally primitive bicirculants, {\it J. London Math. Soc.} {\bf 113} (2026), 1--32.

\bibitem{Kovacs-2004}
I. Kov\'acs, Classifying arc-transitive circulants, {\it J. Algebraic Combin.} {\bf 20} (2004), 353--358.

\bibitem{LCH-circulant-2005}
C. H. Li, Permutation groups with a cyclic regular subgroup and arc-transitive circulants, {\it J. Algebraic Combin.} {\bf 21} (2005), 131--136.




\bibitem{LPVZ-2002}
C. H. Li, C. E. Praeger, A. Venkatesh and S. M. Zhou, Finite locally-quasiprimitive graphs, {\it Discrete Math.} {\bf 246} (2002), 197--218.

\bibitem{LXZ-at-circ-2021}
C. H. Li, B. Z. Xia and S. M. Zhou, An explicit characterization of arc-transitive circulants, {\it J. Combin. Theory Ser. B} {\bf 150} (2021), 1--16.

\bibitem{LPS-1}
M. W. Liebeck, C. E. Praeger and J. Saxl, On the O'Nan-Scott theorem for finite primitive permutation groups, {\it J. Aust. Math. Soc. Ser. A} {\bf 44} (1988), 389--396.

\bibitem{Praeger-1993-onanscott}
C. E. Praeger, An O'Nan-Scott theorem for finite quasiprimitive permutation groups and an application to 2-arc-transitive graphs, {\it J. London Math. Soc. (2)} {\bf 47} (1993), 227--239.

\bibitem{Praeger-2}
C. E. Praeger, Finite transitive permutation groups and finite vertex-transitive graphs, in {\it Graph Symmetry: Algebraic Methods and Applications} (Montreal, 1996), NATO Adv. Sci. Inst. Ser. C Math. Phys. Sci. {\bf 497}, Kluwer Acad. Publ., Dordrecht, 1997, pp. 277--318.

\bibitem{Praeger-2003-biq}
C. E. Praeger, Finite transitive permutation groups and bipartite vertex-transitive graphs, {\it Illinois J. Math.} {\bf 47} (2003), 461--475.

\bibitem{Wielandt-book}
H. Wielandt, {\it Finite Permutation Groups}, Academic Press, New York, 1964.

\bibitem{Xu98}
M. Y. Xu, Automorphism groups and isomorphisms of Cayley digraphs, {\it Discrete Math.} {\bf 182} (1998), 309--319.

\end{thebibliography}
\end{document}